\documentclass[12pt]{article}
\usepackage{amsfonts}

\newtheorem{theorem}{Theorem}[section]

\newtheorem{corollary}[theorem]{Corollary}

\newtheorem{example}[theorem]{Example}

\newtheorem{lemma}[theorem]{Lemma}

\newenvironment{proof}[1][Proof]{\noindent\textbf{#1.} }{\ \rule{0.5em}{0.5em}}
\input{tcilatex}

\begin{document}

\title{Intermediate algebras in Archimedean semiprime $f$-algebras}
\author{Karim Boulabiar \\
{\small D\'{e}partement de Mathematiques, Facult\'{e} des Sciences de Tunis}%
\\
{\small Universit\'{e} de Tunis El Manar, 2092, El Manar, Tunisia}\\
{\small Mathematisch Instituut, Leiden University, Nederland}\\
\texttt{\small karim.boulabiar@fst.utm.tn}}
\maketitle

\begin{abstract}
We introduce the notion of bounded quasi-inversion closed semiprime $f$%
-algebras and we prove that, if $A\ $is such an algebra, then any
intermediate algebra in $A$ is an order ideal of $A$. This extends a recent
result by Domn\'{\i}guez who has dealt with the unital case (the problem on $%
C\left( X\right) $-type spaces has been solved earlier by Dom\'{\i}nguez, G%
\'{o}mez-P\'{e}rez, and Mulero). To illustrate our main result, we show that
any subalgebra of $L^{0}\left( 
{\mu}%
\right) $ with $%
{\mu}%
$ $\sigma $-finite (respectively, $C^{\infty }\left( K\right) $ with $K$
Stonean) containing $L^{\infty }\left( 
{\mu}%
\right) $ (respectively, $C\left( K\right) $) is automatically a Dedekind
complete $f$-algebra with unit.
\end{abstract}

\bigskip 

\section{Introduction}

Let $X$ be a topological space. The lattice-ordered algebra of all
real-valued continuous functions on $X\ $is denoted by $C\left( X\right) $
and its subalgebra of all bounded functions by $C_{b}\left( X\right) $. In
their interesting paper \cite{DGM97}, Dom\'{\i}nguez, G\'{o}mez-P\'{e}rez,
and Mulero define an intermediate algebra in $C\left( X\right) $ to be any
subalgebra of $C\left( X\right) $ containing $C_{b}\left( X\right) $. Among
other facts, they proved that an intermediate algebra in $C\left( X\right) $
is an order ideal (i.e., a solid vector subspace) of $C\left( X\right) $.
Recently, this remarkable result has been extended by Dom\'{\i}nguez in \cite%
{D24} to the more general setting of Archimedean unital $f$-algebras ($\Phi $%
-algebras in the terminology of Henriksen and Johnson in \cite{HJ61}) as
follows.

Let $A$ be an Archimedean $f$-algebra with a multiplicative unit $e$ and
assume that $A$ is bounded inversion closed, i.e., $a\geq e$ in $A$ implies $%
a$ has an inverse in $A$ (this concept is due again to Henriksen and Johnson
in \cite{HJ61}). The set of all bounded elements in $A$ is denoted by $A_{b}$%
, i.e.,%
\[
A_{b}=\left\{ a\in A:\left\vert a\right\vert \leq 
{\mu}%
e\text{ for some }%
{\mu}%
\in \left( 0,\infty \right) \right\} 
\]%
(notice that $A_{b}$ is just the principal order ideal of $A\ $generated by $%
e$). Then any intermediate algebra in $A$ (i.e., a subalgebra of $A\ $%
containing $A_{b}$) is an order ideal of $A$ (for more details, see \cite[%
Proposition 4]{D24}).

The main purpose of this paper is to extend Dom\'{\i}nguez's theorem to the
wider class of Archimedean semiprime $f$-algebras then to provide some
applications mainly in concrete situations. In order to get around the lack
of multiplicative units, we thought about using the notion of
quasi-invertible elements (see, e.g., \cite{BD73}), which is valid in the
non unital case. We then introduced a \textquotedblleft
new\textquotedblright\ concept, namely, bounded quasi-inversion closedness,
to replace the notion of bounded inversion closedness. Further details seem
to be in order.

Let $A$ be an Archimedean $f$-algebra and assume that $A\ $is semiprime,
i.e., with no nonzero nilpotent elements. Then $A\ $is said to be bounded
quasi-inversion closed if, for every $a\in A$, the inequality $\left\vert
a\right\vert \leq \left\vert a^{2}-a\right\vert $ implies that $a$ has a
quasi-inverse in $A$ (see \cite{BD73} for quasi-invertibility in an
arbitrary associative real algebra). We first prove that this the two
closedness notions coincide in the unital case. Also, we define the set $%
A_{b}$ of all bounded elements in $A$ by $A_{b}=\left\{ a\in A:a^{2}\leq 
{\mu}%
\left\vert a\right\vert \text{ for some }%
{\mu}%
\in \left( 0,\infty \right) \right\} $. We prove that $A_{b}$ is an order
ideal and a subalgebra of $A$. We can thus extend the notion of intermediate
algebras in $A$ in an obvious way. We prove that if $A$ is bounded
quasi-inverison closed then any intermediate algebra in $A$ is an order
ideal. We show also that if $A$ is relatively uniformly complete then $A\ $%
bounded quasi-inversion closed and so any intermediate algebra in $A$ is
again an order ideal of $A$. As application, we show that if $L$ is a
Dedekind complete vector lattice with a distinguished weak order unit $e$,
and if its universally complete $L^{u}$ is equipped with its unique
structure of an $f$-algebra with $e$ as multiplicative unit, then any
subalgebra of $L^{u}$ containing the principal order ideal of $L$ generated
by $e$ is a Dedekind complete $f$-algebra with $e$ as multiplicative unit.
This result can be illustrated as follows. Any subalgebra of $L^{0}\left( 
{\mu}%
\right) $ with $%
{\mu}%
$ $\sigma $-finite (respectively, $C^{\infty }\left( K\right) $ with $K$
Stonean) containing $L^{\infty }\left( 
{\mu}%
\right) $ (respectively, $C\left( K\right) $) is automatically a Dedekind
complete $f$-algebra with unit.

We take it for granted that the reader is familiar with the Theory of Vector
Lattices. In this regard, we refer to the classical references \cite%
{AB03,LZ71} for either unexplained terminology and notation or unproved
properties and results.

\section{Bounded quasi-inversion closedness}

The real associative algebra $A$ is called a \emph{lattice-ordered algebra}
if $A$ is simultaneously a vector lattice whose positive cone $A_{+}=\left\{
a\in A:0\leq a\right\} $ is closed under multiplication, i.e.,%
\[
ab\in A_{+}\text{ for all }a,b\in A_{+}.
\]%
The lattice-ordered algebra $A$ is called an $f$\emph{-algebra }if $A$
satisfies the extra condition that%
\[
a\wedge b=0\text{ in }A\text{ implies }ac\wedge b=ca\wedge b=0\text{ for all 
}c\in A_{+}.
\]%
An $f$-algebra $A$ is said to be \emph{semiprime} if $0$ is the only
nilpotent element in $A$. Most of the classical real algebras of functions
are examples of semiprime $f$-algebras. We now record some elementary
properties of $f$-algebras (for proofs, see \cite{HP82,P81,Z83}).

\begin{enumerate}
\item Any Archimedean $f$-algebra is commutative.

\item Multiplications by positive elements in any $f$-algebra $A$ are
lattice homomorphisms. Hence, if $a,b\in A$ with $b\in A_{+}$ then $%
a^{+}b=\left( ab\right) ^{+}$ and $ba^{+}=\left( ba^{+}\right) $.

\item The equality $\left\vert ab\right\vert =\left\vert a\right\vert
\left\vert b\right\vert $ holds in any $f$-algebra.

\item If $a$ is an element in an $f$-algebra then $a^{+}a^{-}=a^{-}a^{+}=0$.

\item Squares in an $f$-algebra $A$ are positive. More precisely, $%
a^{2}=\left\vert a\right\vert ^{2}$ for all $a\in A$.

\item If $A$ is a semiprime $f$-algebra and $a,b\in A$, then $ab=0$ if and
only if $\left\vert a\right\vert \wedge \left\vert b\right\vert =0$.

\item If $A$ is semiprime and $a,b\in A$ then $a^{2}\leq b^{2}$ if and only
if $\left\vert a\right\vert \leq \left\vert b\right\vert $.

\item If an $f$-algebra has a multiplicative unit $e$ then $0\leq e$ and if $%
A$ is, in addition, Archimedean the $A$ is semiprime.
\end{enumerate}

\noindent \textsl{Throughout the paper, }$A$\textsl{\ is an Archimedean
semiprime }$f$\textsl{-algebra. }An element $a\in A$ is said to be \emph{%
quasi-invertible }if the equality%
\[
a+a^{\ast }=aa^{\ast }
\]%
holds for some $a^{\ast }\in A$. It is easily checked that such an element $%
a^{\ast }$ ---if one exists--- is unique, referred to as the \emph{%
quasi-inverse} of $a$. The set of all quasi-invertible elements in $A$ is
denoted by $Q\left( A\right) $. Now, $A$ is said to be \emph{bounded
quasi-inversion closed} if, for every $a\in A$, we have%
\[
\left\vert a\right\vert \leq \left\vert a^{2}-a\right\vert \text{ implies }%
a\in Q\left( A\right) .
\]%
The following lemma will be very useful later.

\begin{lemma}
\label{negative}Assume that $A$ is bounded quasi-inversion closed and let $%
a\in A$. If $a\leq 0$ then $a\in Q\left( A\right) $ and $0\leq a^{\ast }$.
\end{lemma}

\begin{proof}
Let $a\in A_{+}$ and observe that%
\[
\left\vert \left( -a\right) ^{2}-\left( -a\right) \right\vert =a^{2}+a\geq
a=\left\vert -a\right\vert .
\]%
It follows that $-a\in Q\left( A\right) $ because $A$ is bounded
quasi-inversion closed. For brevity, put $b=\left( -a\right) ^{\ast }$.
Multiplying the equality $ba=a-b$ by  $b^{-}$, we get%
\[
-\left( b^{-}\right) ^{2}a=b^{-}a+\left( b^{-}\right) ^{2}.
\]%
Therefore,%
\[
0\leq \left( b^{-}\right) ^{2}\leq b^{-}a+\left( b^{-}\right) ^{2}=-\left(
b^{-}\right) ^{2}a\leq 0.
\]%
Consequently, $\left( b^{-}\right) ^{2}=0$ and so $b^{-}=0$ because $A$ is
semiprime. This shows that $b\in A_{+}$ and completes the proof of the lemma.
\end{proof}

Recall that if $A$ has a multiplicative unit $e$, then $A\ $is said to be 
\emph{bounded inversion closed} if $a\in A$ has an inverse in $A$ whenever $%
e\leq a$. Next, we shall see that bounded quasi-inversion closedness and
bounded inversion closedness coincide in the unital case.

\begin{theorem}
\label{=}If $A$ is unital then $A$ is bounded quasi-inversion closed if and
only if $A$ is bounded inversion closed.
\end{theorem}

\begin{proof}
Let $e$ denote the multiplicative unit of $A$. Assume that $A$ is bounded
quasi-inversion closed and pick $a\in A$ with $a\geq e$. Since $a-e\geq 0$,
Lemma \ref{negative} tells us that $e-a\in Q\left( A\right) $ and so the
equality%
\[
\left( e-a\right) c=e-a+c
\]%
holds for some $c\in A$. Therefore, $a\left( e-c\right) =e$ and so $a$ has
an inverse in $A$. This means that $A$ is bounded inversion closed.

Conversely, assume that $A$ is bounded inversion closed and choose $a\in A$
such that $\left\vert a\right\vert \leq \left\vert a^{2}-a\right\vert $.
Hence,%
\[
\left( \left\vert e-a\right\vert -e\right) \left\vert a\right\vert \geq 0
\]%
from which it follows that%
\[
\left( \left\vert e-a\right\vert -e\right) ^{-}\left\vert a\right\vert =0.
\]%
Moreover,%
\[
\left( \left\vert e-a\right\vert -e\right) ^{-}\leq \left\vert \left\vert
e-a\right\vert -e\right\vert \leq \left\vert a\right\vert .
\]%
Accordingly,%
\[
0\leq \left( \left( \left\vert e-a\right\vert -e\right) ^{-}\right) ^{2}\leq
\left( \left\vert e-a\right\vert -e\right) ^{-}\left\vert a\right\vert =0
\]%
and thus $\left( \left\vert e-a\right\vert -e\right) ^{-}=0$. We derive that 
$\left\vert e-a\right\vert \geq e$. But then $\left\vert e-a\right\vert $ is
invertible in $A$ (and so is $e-a$) as $A$ is bounded inversion closed.
There exists therefore $b\in A$ such that $\left( e-a\right) b=e$.
Consequently,%
\[
a\left( e-b\right) =a-ab=a+\left( e-b\right) ,
\]%
which yields that $a\in Q\left( A\right) $ and then $A$ is bounded
quasi-inversion closed.
\end{proof}

For every $a\in A$, an operator $\pi _{a}$ can be defined on $A$ by%
\[
\pi _{a}\left( x\right) =ax\text{ for all }x\in A.
\]%
It follows directly from the definition of an $f$-algebra that%
\[
\pi _{a}\in \mathrm{Orth}\left( A\right) \text{ for all }a\in A,
\]%
where $\mathrm{Orth}\left( A\right) $ denotes the unital Archimedean $f$%
-algebra of all orthomorphisms on $A$. Moreover, the operator $\pi $ from $A$
into $\mathrm{Orth}\left( A\right) $ given by%
\[
\pi \left( a\right) =\pi _{a}\text{ for all }a\in A
\]%
is an injective algebra and lattice homomorphism. Accordingly, identifying
any $a\in A$ with the orthomorphism $\pi _{a}$, $A$ can be regarded as a
subalgebra and a vector sublattice of $A$. Furthermore, it is not hard to
see that, for every $a\in A$, the equality $\pi \pi _{a}=\pi _{\pi \left(
a\right) }$ holds and so $A$ is even a ring ideal of $\mathrm{Orth}\left(
A\right) $. Lastly, if $A$ is relatively uniformly complete, then so is $%
\mathrm{Orth}\left( A\right) $ (see \cite{LZ71} for relatively uniform
completeness). The reader can consult the books \cite{BKW77,Z83} to have
more detailed information about orthomorphisms. However, we do believe that
the manuscript \cite{P81} remains the most complete reference on the subject.

\begin{theorem}
\label{ru}If $A$ is relatively uniformly complete then $A$ is bounded
quasi-inversion closed.
\end{theorem}

\begin{proof}
Assume that $A$ is uniformly complete. Choose $a\in A$ such that $\left\vert
a\right\vert \leq \left\vert a^{2}-a\right\vert $ in $A$. Hence, the
inequality%
\[
0\leq \left( \left\vert I-a\right\vert -I\right) \left\vert a\right\vert 
\]%
holds in $\mathrm{Orth}\left( A\right) $, where $I$ denotes the identity
operator on $A$ (which is the multiplicative unit of $\mathrm{Orth}\left(
A\right) $). It follows that%
\[
\left( \left\vert I-a\right\vert -I\right) ^{-}\left\vert a\right\vert =0.
\]%
Moreover,%
\[
\left( \left\vert I-a\right\vert -I\right) ^{-}\leq \left\vert \left\vert
I-a\right\vert -I\right\vert \leq \left\vert a\right\vert .
\]%
Therefore,%
\[
0\leq \left( \left( \left\vert I-a\right\vert -I\right) ^{-}\right) ^{2}\leq
\left( \left\vert I-a\right\vert -I\right) ^{-}\left\vert a\right\vert =0
\]%
and thus $\left( \left\vert I-a\right\vert -I\right) ^{-}=0$. In other
words, $I\leq \left\vert I-a\right\vert $ in $\mathrm{Orth}\left( A\right) $%
. But then $\left\vert I-a\right\vert $ has an inverse in $\mathrm{Orth}%
\left( A\right) $ because, being relatively uniformly complete, $\mathrm{Orth%
}\left( A\right) $ is bounded inversion closed (Theorem 3.4 in \cite{HP82}).
Consequently, $I-a$ has an inverse in $\mathrm{Orth}\left( A\right) $, so
there exists $\pi \in \mathrm{Orth}\left( A\right) $ such that $\left(
I-a\right) \pi =I$. This yields that $I-\pi =-a\pi $ and thus $I-\pi \in A$
because $A$ is a ring ideal in $\mathrm{Orth}\left( A\right) $. Put $b=I-\pi 
$ and observe that%
\[
ab=a\left( I-\pi \right) =a-a\pi =a+I-\pi =a+b.
\]%
This shows that $A$ is bounded quasi-inversion closed and completes the
proof of the theorem.
\end{proof}

On closer examination of the proof of Theorem \ref{ru}, it can be deduced
that if $A$ is unital and bounded inversion closed, then any vector
sublattice of $A$ which is simultaneously a ring ideal of $A$ is bounded
quasi-inversion closed as a subalgebra of $A$. This fact extends Sufficiency
in Theorem \ref{=}.

\section{Intermediate algebras and applications}

\begin{quote}
\textsl{Yet again, }$A$ \textsl{stands for an Archimedean semiprime }$f$%
\textsl{-algebra.}
\end{quote}

\noindent Our investigation starts with the following elementary properties
which will come in handy later on.

\begin{lemma}
\label{elem}The following hold for all $a,b,c\in A$.

\begin{enumerate}
\item[\emph{(i)}] If $ax\leq bx$ for all $x\in A_{+}$ then $a\leq b$.

\item[\emph{(ii)}] If $a^{2}\left\vert b\right\vert \leq \left\vert
ac\right\vert $ then $\left\vert ab\right\vert \leq \left\vert c\right\vert $%
.
\end{enumerate}
\end{lemma}

\begin{proof}
$\mathrm{(i)}$ If $a\leq b$ and $x\in A_{+}$ then, of course, $ax\leq bx$.
Conversely, assume that $ax\leq bx$ for all $x\in A_{+}$. Hence, of every $%
x\in A_{+}$, we have%
\[
\left( a-b\right) ^{+}x=\left( ax-bx\right) ^{+}=0. 
\]%
In particular, $\left( \left( a-b\right) ^{+}\right) ^{2}=0$ and so $\left(
a-b\right) ^{+}=0$ as $A$ is semiprime This yields that $a\leq b$ and
completes the proof of $\mathrm{(i)}$.

$\mathrm{(ii)}$ Assume that $a^{2}\left\vert b\right\vert \leq \left\vert
ac\right\vert $ and observe that%
\[
0\leq \left( \left( \left\vert ab\right\vert -\left\vert c\right\vert
\right) ^{+}\right) ^{2}\leq \left\vert ab\right\vert \left( \left\vert
ab\right\vert -\left\vert c\right\vert \right) ^{+}=\left\vert b\right\vert
\left( a^{2}\left\vert b\right\vert -\left\vert ac\right\vert \right) ^{+}=0.
\]%
But then $\left( \left\vert ab\right\vert -\left\vert c\right\vert \right)
^{+}=0$ because $A$ is semiprime and so $\left\vert ab\right\vert \leq
\left\vert c\right\vert $, which is the required inequality.
\end{proof}

The principal order ideal of $A$ generated by an element $a\in A$ is denoted
by $A\left( a\right) $, that is,%
\[
A\left( a\right) =\left\{ x\in A:\left\vert x\right\vert \leq 
{\mu}%
\left\vert a\right\vert \text{ for some }%
{\mu}%
\in \left( 0,\infty \right) \right\} 
\]%
An element $a\in A$ is said to be \emph{bounded} if $a^{2}\in A\left(
a\right) $. The set of all bounded elements in $A$ is denoted by $A_{b}$. An
alternative characterization of bounded elements in $A$ is given next.

\begin{lemma}
\label{Pagter}If $a$ is an element in $A$ then $a\in A_{b}$ if and only if%
\[
ax\in A\left( x\right) \text{ for all }x\in A.
\]
\end{lemma}

\begin{proof}
The \textquotedblleft if\textquotedblright\ part being trivial, we prove the
\textquotedblleft only if\textquotedblright\ part. So, suppose that $a\in
A_{b}$ and pick $%
{\mu}%
\in \left( 0,\infty \right) $ such that $a^{2}\leq 
{\mu}%
\left\vert a\right\vert $. Hence, if $x\in A_{+}$ then $a^{2}x\leq 
{\mu}%
\left\vert a\right\vert x$. Using conveniently Lemma \ref{elem} $\mathrm{(ii)%
}$, we get $\left\vert a\right\vert x\leq 
{\mu}%
x$ and we are done.
\end{proof}

The structure of $A_{b}$ is given in what follows.

\begin{theorem}
\label{structure}The set $A_{b}$ is simultaneously a subalgebra and an order
ideal of $A$.
\end{theorem}

\begin{proof}
Clearly, $0\in A_{b}$ and so $A_{b}$ is nonempty. Let $a,b\in A_{b}$ and $%
\delta \in \mathbb{R}$. Clearly, from Lemma \ref{Pagter} it follows that
there exists $%
{\mu}%
\in \left( 0,\infty \right) $ such that%
\[
\left\vert a\right\vert x\leq 
{\mu}%
x\ \text{and }\left\vert b\right\vert x\leq 
{\mu}%
x\quad \text{for all }x\in A_{+}.
\]%
Hence, if $x\in A_{+}$ then%
\[
0\leq \left\vert a+\delta b\right\vert x\leq \left\vert a\right\vert
x+\left\vert \delta \right\vert \left\vert b\right\vert x\leq \left(
1+\left\vert \delta \right\vert \right) 
{\mu}%
x.
\]%
We derive that $a+\delta b\in A_{b}$ and so $A_{b}$ is a vector subspace of $%
A$. Now, if $x\in X$ then%
\[
\left\vert ab\right\vert x=\left\vert a\right\vert \left\vert b\right\vert
x\leq 
{\mu}%
\left\vert b\right\vert x\leq 
{\mu}%
^{2}x.
\]%
This means that $ab\in A_{b}$ from which it follows that $A_{b}$ is a
subalgebra of $A$. Finally, assume that $\left\vert c\right\vert \leq
\left\vert a\right\vert $ for some $c\in A$. Hence,%
\[
\left\vert c\right\vert x\leq \left\vert a\right\vert x\leq 
{\mu}%
x\quad \text{for all }x\in A_{+}.
\]%
and so $c\in A_{b}$. Accordingly, $A_{b}$ is an order ideal of $A$ and the
proof is complete.
\end{proof}

It should be pointed out that a look at bounded elements in an Archimedean
semiprime $f$-algebra from a different point of view can be found in \cite%
{T00} by Triki.

In spite of its simplicity, the following fact which plays a key role in the
proof of the main theorem of this paper.

\begin{lemma}
\label{square}If $a\in A$ then $a\in A_{b}$ if and only if $a^{2}\in A_{b}$.
\end{lemma}

\begin{proof}
Theorem \ref{structure} yields that if $a\in A_{b}$ then $a^{2}\in A_{b}$.
Conversely, suppose that $a^{2}\in A_{b}$, that is, $a^{4}\in A\left(
a^{2}\right) $. Thus, there exists $%
{\mu}%
\in \left( 0,\infty \right) $ such that $a^{4}\leq 
{\mu}%
a^{2}$. Whence,%
\[
\left( a^{2}\right) ^{2}=a^{4}\leq 
{\mu}%
a^{2}=\left( \sqrt{%
{\mu}%
}\left\vert a\right\vert \right) ^{2}.
\]%
It follows $a^{2}\leq \sqrt{%
{\mu}%
}\left\vert a\right\vert $ and the lemma follows.
\end{proof}

Following Rodriguez in \cite{D24}, we call a subalgebra of $A$ containing $%
A_{b}$ an \emph{intermediate algebra }in $A$. We are in position at this
point to state and prove the central theorem of this study.

\begin{theorem}
\label{main}Assume that $A$ is \textit{bounded quasi-inversion closed. Then
any intermediate algebra in }$A$ is an order ideal in $A$.
\end{theorem}

\begin{proof}
Let $a\in A$ and $b\in B$ such that $\left\vert a\right\vert \leq \left\vert
b\right\vert $. By Lemma \ref{negative}, $-a^{2}\in Q\left( A\right) $ and
so, to simplify matters, we can put $c=\left( -a^{2}\right) ^{\ast }$.
First, we claim that $b-bc\in A_{b}$. To this end, observe that%
\[
b^{2}c\leq b^{2}c+c=b^{2}
\]%
because $c\geq 0$ (see Lemma \ref{negative}). Using Lemma \ref{elem} $%
\mathrm{(ii)}$, we obtain $\left\vert b\right\vert c\leq \left\vert
b\right\vert $ and so $0\leq \left\vert b\right\vert -\left\vert
b\right\vert c$. Hence, if $x\in A_{+}$ then%
\begin{eqnarray*}
\left\vert b-bc\right\vert x &=&\left\vert bx-bxc\right\vert =\left\vert
b\right\vert \left\vert x-xc\right\vert  \\
&=&\left\vert \left\vert b\right\vert x-\left\vert b\right\vert
xc\right\vert =\left\vert \left\vert b\right\vert -\left\vert b\right\vert
c\right\vert x=\left( \left\vert b\right\vert -\left\vert b\right\vert
c\right) x
\end{eqnarray*}%
As $x$ is arbitrary in $A_{+}$, we derive from Lemma \ref{elem} $\mathrm{(i)}
$ that%
\[
\left\vert b-bc\right\vert =\left\vert b\right\vert -\left\vert b\right\vert
c.
\]%
Moreover,%
\[
b^{2}\left\vert b-bc\right\vert =\left\vert b^{3}-b^{3}c\right\vert
=\left\vert b\left( b^{2}-b^{2}c\right) \right\vert =\left\vert
bc\right\vert =\left\vert b\right\vert c\leq \left\vert b\right\vert .
\]%
Thus,%
\[
b^{2}\left\vert b-bc\right\vert x\leq \left\vert b\right\vert x\quad \text{%
for all }x\in A_{+}.
\]%
This together with Lemma \ref{elem} $\mathrm{(ii)}$ yields that%
\[
\left\vert b-bc\right\vert ^{2}x\leq \left\vert b\right\vert \left\vert
b-bc\right\vert x\leq x\quad \text{for all }x\in A_{+}.
\]%
Combining Lemmas \ref{Pagter} and \ref{square}, we infer that $b-bc\in A_{b}$%
, as required.

Now, let $x\in A_{+}$ and observe that%
\begin{eqnarray*}
\left\vert a-ac\right\vert x &=&\left\vert ax-acx\right\vert =\left\vert
a\right\vert \left\vert x-cx\right\vert  \\
&\leq &\left\vert b\right\vert \left\vert x-cx\right\vert =\left\vert
bx-bcx\right\vert =\left\vert b-bc\right\vert x.
\end{eqnarray*}%
It follows that $\left\vert a-ac\right\vert \leq \left\vert b-bc\right\vert $%
, where we use Lemma \ref{elem} $\mathrm{(i)}$. As $b-bc\in A_{b}$ and $A_{b}
$ is an order ideal of $A$ (see Theorem \ref{structure}), we derive that%
\[
a-ac\in A_{b}\subset B.
\]%
Finally,%
\[
a=a+\left( b^{2}-b^{2}c-c\right) a=\left( a-ac\right) +b^{2}\left(
a-ac\right) \in B
\]%
and the proof is complete.
\end{proof}

Theorem \ref{main} together with Theorem \ref{ru} leads straightforwardly to
the following result.

\begin{corollary}
\label{ruo}If $A\ $relatively uniformly complete then \textit{any
intermediate algebra in }$A$ is an order ideal in $A$.
\end{corollary}

Taking into account that any Dedekind complete vector lattice is relatively
uniformly complete, the following can be directly derived from Corollary \ref%
{ruo}.

\begin{corollary}
If $A\ $Dedekind complete then \textit{any intermediate algebra in }$A$ is
an order ideal in $A$. In particular, any intermediate algebra in $A$ is a
Dedekind complete semiprime $f$-algebra.
\end{corollary}

Now, let $L$ be a Dedekind complete vector lattice with a distinguished weak
order unit $e$. It is well-known that that there exists a unique
multiplication on the universal completion $L^{u}$ of $L$ which turns $L^{u}$
into a universal complete $f$-algebra with $e$ as multiplicative unit (see 
\cite{AB03} or \cite{L79}). We obtain the following.

\begin{corollary}
\label{last}Let $L$ be a Dedekind complete vector lattice with a
distinguished weak order unit $e$. Any subalgebra of $L^{u}$ containing the
principal order ideal $L\left( e\right) $ of $L$ generated by $e$ is a
Dedekind complete $f$-algebra with $e$ as a multiplicative unit.
\end{corollary}

\begin{proof}
Let $B$ be a subalgebra of $L^{u}$ and suppose that $B$ contains $L\left(
e\right) $. Clearly, $L\left( e\right) $ is contained in $L^{u}\left(
e\right) $, the principal order ideal of $L^{u}$ generated by $e$.
Conversely, if $x\in L^{u}\left( e\right) $ then $x\in L\left( e\right) $
because $L$ is an order ideal in $L^{u}$. This means that $B$ is an
intermediate algebra in $L^{u}$ and the result follows from the previous
corollary.
\end{proof}

We end by discussing two (equivalent) concrete illustrations.

\begin{example}

\begin{enumerate}
\item[\emph{(i)}] Let $K$ be an extremally disconnected compact Hausdorff
space. Hence, $C\left( K\right) $ is a Dedekind complete vector lattice with
the constant function $1$ as a strong \emph{(}and so weak\emph{)} order
unit. The universal completion of $C\left( K\right) $ is the universally
complete unital $f$-agebra $C^{\infty }\left( K\right) $ of all
almost-finite extended-real valued continuous functions\textsl{\ }on on $K$.
According to \emph{Corollary \ref{last}}, any subalgebra of $C^{\infty
}\left( K\right) $ containing $C\left( K\right) $ is a Dedekind complete $f$%
-algebra with the constant function $1$ as multiplicative unit.

\item[\emph{(ii)}] Let $\left( \Omega ,\Sigma ,%
{\mu}%
\right) $ be a measure space such that the measure $%
{\mu}%
$ is $\sigma $-finite and let $L^{0}\left( 
{\mu}%
\right) $ denote the universally complete unital $f$-algebra of all $%
{\mu}%
$-measurable real-valued \emph{(}classes of\emph{)} functions on $\Omega $.
It is know that $L^{0}\left( 
{\mu}%
\right) $ is the universal completion of the Dedekind complete vector
lattice $L^{\infty }\left( 
{\mu}%
\right) $ of all bounded \emph{(}classes of\emph{)} functions on $\Omega $.
Moreover, the constant function $1$ is a weak order unit in $L^{\infty
}\left( 
{\mu}%
\right) $ and, in the same time, the multiplicative identity of $L^{0}\left( 
{\mu}%
\right) $. In view of \emph{Corollary \ref{last}}, any subalgebra of $%
L^{0}\left( 
{\mu}%
\right) $ containing $L^{\infty }\left( 
{\mu}%
\right) $ is a Dedekind complete $f$-algebra with the constant function $1$
as multiplicative unit.
\end{enumerate}
\end{example}

\noindent \textbf{Acknowledgment.} The author started working on this
problem after a question asked by Eugene Bilokopytov at a recent workshop in
Leiden University followed by a fruitful discussion with him and Marten
Wortel to whom thanks are addressed.

\end{document}